\documentclass{amsart}
\usepackage{amsmath}
\usepackage{amssymb}
\usepackage{amsthm}
\usepackage{epsfig}
\usepackage{amsfonts}
\usepackage{amscd}
\usepackage{mathrsfs}
\usepackage{enumerate}
\usepackage{latexsym}
\usepackage{url}

\newtheorem{theorem}{Theorem}[section]
\newtheorem{lemma}[theorem]{Lemma}

\newtheorem{xclaim}[theorem]{Claim}

\theoremstyle{definition}
\newtheorem{definition}[theorem]{Definition}
\newtheorem{project}[theorem]{Project}
\newtheorem{example}[theorem]{Example}
\newtheorem{notation}[theorem]{Notation}

\theoremstyle{remark}
\newtheorem{remark}[theorem]{Remark}
\numberwithin{equation}{section}

\begin{document}

\title[One-point extensions of locally compact paracompact spaces]{One-point extensions of locally compact paracompact spaces}

\author{M.R. Koushesh}
\address{Department of Mathematical Sciences, Isfahan University of Technology, Isfahan 84156--83111, Iran}
\address{School of Mathematics, Institute for Research in Fundamental Sciences (IPM), P.O. Box: 19395--5746, Tehran, Iran}
\email{koushesh@cc.iut.ac.ir}
\thanks{This research was in part supported by a grant from IPM (No. 86540012).}



\begin{abstract}
A space $Y$ is called an {\em extension} of a space $X$ if $Y$
contains $X$ as a dense subspace.
Two extensions of $X$ are said to be {\em equivalent} if there is a homeomorphism between them which fixes $X$ point-wise.
For two (equivalence classes of)  extensions $Y$ and $Y'$ of $X$ let $Y\leq Y'$ if there is a continuous function of $Y'$ into $Y$
which fixes  $X$ point-wise. An extension $Y$ of $X$ is called a {\em one-point extension}  if $Y\backslash X$ is a singleton.
An extension $Y$ of  $X$ is called {\em first-countable} if $Y$ is first-countable at points of $Y\backslash X$.
Let ${\mathcal P}$ be a topological
property. An extension $Y$ of $X$ is called a {\em
${\mathcal P}$-extension} if it has ${\mathcal P}$.

In this article, for a given locally compact paracompact space $X$, we consider the two classes of one-point \v{C}ech-complete ${\mathcal P}$-extensions of $X$  and one-point first-countable locally-${\mathcal P}$ extensions of $X$, and we study their order-structures, by relating them to the  topology of a certain subspace of the outgrowth $\beta X\backslash X$. Here ${\mathcal P}$
is subject to some requirements and include $\sigma$-compactness and the Lindel\"{o}f property as special cases.
\end{abstract}

\maketitle

\section{Introduction}

A space $Y$ is called an {\em extension} of a space $X$ if $Y$
contains $X$ as a dense subspace. If $Y$ is an extension of   $X$ then the subspace $Y\backslash X$ of $Y$ is called the
{\em  remainder} of $Y$.  Extensions with a one-point remainder are called {\em one-point extensions}. Two extensions of  $X$ are said to be {\em equivalent} if there exists a homeomorphism between them
which fixes $X$ point-wise. This defines an equivalence relation on the class of all extensions of  $X$. The equivalence classes
will be identified with individuals when this causes no confusion. For two  extensions $Y$ and $Y'$ of $X$ we
let $Y\leq Y'$ if there exists a continuous function of $Y'$ into $Y$ which fixes  $X$ point-wise. The relation $\leq$ defines  a partial
order on the set of extensions of $X$ (see Section 4.1 of \cite{PW} for more details). An extension $Y$ of  $X$ is called {\em first-countable} if $Y$ is first-countable at points of $Y\backslash X$, that is, $Y$ has a countable local base at every point of $Y\backslash X$. Let ${\mathcal P}$ be a topological
property. An extension $Y$ of $X$ is called a {\em
${\mathcal P}$-extension} if it has ${\mathcal P}$. If ${\mathcal P}$ is compactness then
${\mathcal P}$-extensions are called  {\em compactifications}.

This work was mainly  motivated by our previous work \cite{Ko2} (see \cite{Be}, \cite{HJW}, \cite{Ko1}, \cite{Ko4},  \cite{MRW1} and \cite{MRW2} for related results) in which we have studied the partially ordered set of one-point ${\mathcal P}$-extensions of a given locally compact space $X$ by relating it to the topologies of certain subspaces of its outgrowth  $\beta X\backslash X$. In this article we continue our studies by considering the classes of one-point \v{C}ech-complete ${\mathcal P}$-extensions and one-point first-countable locally-${\mathcal P}$ extensions of a given locally compact paracompact space $X$. The topological property ${\mathcal P}$
is subject to some requirements and include $\sigma$-compactness, the Lindel\"{o}f property and the linearly Lindel\"{o}f property as special cases.

We  review some of the terminology, notation and well-known results that will be used in the sequel. Our definitions mainly come from the standard  text
\cite{E} (thus, in particular, compact spaces are Hausdorff,  etc.). Other useful sources are \cite{GJ} and \cite{PW}.

The letters $\mathbf{I}$ and $\mathbf{N}$ denote
the closed unit interval and the set of all positive integers, respectively.  For a subset $A$ of a space $X$ we let $\mbox{cl}_X
A$ and $\mbox{int}_X A$  denote the closure and the
interior  of  $A$ in $X$, respectively.  A subset of a space is called {\em clopen} if it is simultaneously closed and open.
A {\em zero-set} of a space $X$ is a set of the form
$Z(f)=f^{-1}(0)$ for some continuous $f:X\rightarrow \mathbf{I}$.  Any set of the form
$X\backslash Z$, where $Z$ is a zero-set of $X$, is called a {\em
cozero-set} of $X$. We denote the set of all zero-sets of $X$ by
${\mathscr Z}(X)$ and the set of all cozero-sets of $X$ by $Coz(X)$.

For a Tychonoff space $X$ the
{\em Stone-\v{C}ech compactification} of $X$ is the largest (with respect to the partial order $\leq$)
compactification of $X$ and is  denoted by $\beta X$.  The Stone-\v{C}ech compactification of  $X$ can be characterized
among all compactifications of $X$ by either of the following properties:
\begin{itemize}
  \item Every continuous function of $X$ to a compact space is continuously extendible over $\beta X$.
  \item Every continuous function of $X$ to $\mathbf{I}$ is continuously extendible over $\beta X$.
  \item For every $Z,S\in {\mathscr Z}(X)$ we have
  \[\mbox{cl}_{\beta X}(Z\cap S)=\mbox{cl}_{\beta X}Z\cap\mbox{cl}_{\beta X}S.\]
\end{itemize}

A Tychonoff space is called {\em  zero-dimensional} if it has an open base consisting of its clopen subsets. A Tychonoff space
is called {\em strongly zero-dimensional} if its Stone-\v{C}ech compactification is zero-dimensional. A Tychonoff space $X$ is called {\em \v{C}ech-complete} if its outgrowth $\beta X\backslash X$ is an $F_\sigma$ in $\beta X$. Locally compact spaces are  \v{C}ech-complete, and in the realm of metrizable spaces $X$, \v{C}ech-completeness is equivalent to the existence of a compatible complete metric on $X$.

Let ${\mathcal P}$ be a topological property. A  topological space $X$ is called {\em locally-${\mathcal P}$} if for every $x\in X$
there exists an open neighborhood $U_x$ of $x$ in $X$ such that $\mbox{cl}_XU_x$ has ${\mathcal P}$.

A topological property $\mathcal{P}$  is said to be {\em hereditary with respect to closed subsets} if each closed subset of a space with $\mathcal{P}$ also has $\mathcal{P}$. A  topological property ${\mathcal P}$ is said to be  {\em preserved under finite (closed) sums of subspaces}
if a Hausdorff space has ${\mathcal P}$, provided that it is the union of a finite collection of its (closed) ${\mathcal P}$-subspaces.

Let  $(P,\leq)$ and $(Q,\leq)$ be two partially ordered sets. A mapping  $f:(P,\leq)\rightarrow(Q,\leq)$ is said to be an {\em order-homomorphism} ({\em anti-order-homomorphism}, respectively) if $f(a)\leq f(b)$ ($f(b)\leq f(a)$, respectively) whenever $a\leq b$. An order-homomorphism (anti-order-homomorphism, respectively) $f:(P,\leq)\rightarrow(Q,\leq)$ is said to be an {\em order-isomorphism} ({\em anti-order-isomorphism}, respectively) if $f^{-1}:(Q,\leq)\rightarrow(P,\leq)$ (exists and) is an order-homomorphism (anti-order-homomorphism, respectively). Two partially ordered sets $(P,\leq)$ and $(Q,\leq)$ are called  {\em order-isomorphic} ({\em anti-order-isomorphic}, respectively) if there exists an  order-isomorphism (anti-order-isomorphism, respectively) between them.

\section{Motivations, notations and definitions}

In this article  we will be dealing with various sets of one-point extensions of a given topological space $X$.
For the reader's convenience we list these sets all at the beginning.

\begin{notation}\label{JHB}
Let $X$ be a topological space. Denote
\begin{itemize}
  \item ${\mathscr E}(X)=\{Y:Y \mbox{ is a one-point Tychonoff extension of } X\}$
  \item ${\mathscr E}^{\,*}(X)=\{Y\in{\mathscr E}(X):Y \mbox{ is first-countable at } Y\backslash X\}$
  \item ${\mathscr E}^{\,C}(X)=\{Y\in{\mathscr E}(X):Y \mbox{ is \v{C}ech-complete}\}$
  \item ${\mathscr E}^{\,K}(X)=\{Y\in{\mathscr E}(X):Y \mbox{ is locally compact}\}$
\end{itemize}
and when ${\mathcal P}$ is a topological property
\begin{itemize}
  \item ${\mathscr E}_{\,{\mathcal P}}(X)=\{Y\in{\mathscr E}(X) :Y \mbox{ has  } {\mathcal P}\}$
  \item ${\mathscr E}_{\,local-{\mathcal P}}(X)=\{Y\in{\mathscr E}(X) :Y \mbox{ is locally-${\mathcal P}$}\}$.
\end{itemize}
Also, we may use notations which are obtained by combinations of the above notations, e.g.
\[{\mathscr E}^{\,*}_{\,local-{\mathcal P}}(X)={\mathscr E}^{\,*}(X)\cap{\mathscr E}_{\,local-{\mathcal P}}(X).\]
\end{notation}

\begin{definition}[\cite{Ko3}]\label{HGA}
For a Tychonoff space $X$ and a topological property  ${\mathcal P}$, let
\[\lambda_{{\mathcal P}} X=\bigcup\big\{\mbox{int}_{\beta X} \mbox{cl}_{\beta X}C:C\in Coz(X)\mbox{ and } \mbox{cl}_X C \mbox{ has }{\mathcal P}\big\}.\]
\end{definition}

\begin{definition}[\cite{Mr}]\label{HGFB}
We say that a topological
property  $\mathcal{P}$ {\em satisfies Mr\'{o}wka's condition} (W) if it satisfies the following:   If $X$ is a Tychonoff
space in which there exists a point  $p$ with an open  base  ${\mathscr B}$ for $X$ at $p$ such that $X\backslash B$ has $\mathcal{P}$ for each
$B\in {\mathscr B}$, then $X$ has $\mathcal{P}$.
\end{definition}

Mr\'{o}wka's condition (W) is satisfied by a large number of topological properties; among them are (regularity $+$) the  Lindel\"{o}f property, paracompactness, metacompactness, subparacompactness, the para-Lindel\"{o}f property,  the  $\sigma$-para-Lindel\"{o}f  property,  weak $\theta$-refinability,  $\theta$-refinability (or  submetacompactness), weak $\delta\theta$-refinability,  $\delta\theta$-refinability (or the submeta-Lindel\"{o}f property), countable paracompactness, $[\theta,\kappa]$-compact-ness,  $\kappa$-boundedness,  screenability,  $\sigma$-metacompactness,  Dieudonn\'{e} completeness, $N$-compactness \cite{M1},  realcompactness, almost realcompactness \cite{F} and zero-dimensionality (see \cite{Ko3}, \cite{MRW1} and \cite{MRW2} for proofs and \cite{Bu}, \cite{Steph} and \cite{Va} for definitions).

In \cite{Ko4} we have obtained the following result.

\begin{theorem}[\cite{Ko4}]\label{FHFH}
Let $X$ and $Y$ be locally compact locally-${\mathcal P}$ non-${\mathcal P}$  spaces  where  ${\mathcal P}$ is either pseudocompactness or a closed hereditary topological property which is preserved under  finite closed sums of subspaces and satisfies   Mr\'{o}wka's condition {\em (W)}.
The following are equivalent:
\begin{itemize}
\item[\rm(1)] $\lambda_{{\mathcal P}} X\backslash X$ and $\lambda_{{\mathcal P}} Y\backslash Y$ are homeomorphic.
\item[\rm(2)] $({\mathscr E}_{\,{\mathcal P}}(X),\leq)$ and $({\mathscr E}_{\,{\mathcal P}}(Y),\leq)$ are order-isomorphic.
\item[\rm(3)] $({\mathscr E}^{\,C}_{\,{\mathcal P}}(X),\leq)$ and $({\mathscr E}^{\,C}_{\,{\mathcal P}}(Y),\leq)$ are order-isomorphic.
\item[\rm(4)] $({\mathscr E}^{\,K}_{\,{\mathcal P}}(X),\leq)$ and $({\mathscr E}^{\,K}_{\,{\mathcal P}}(Y),\leq)$ are order-isomorphic, provided that  $X$ and $Y$ are moreover strongly zero-dimensional.
\end{itemize}
\end{theorem}

There are topological properties, however, which do not satisfy the assumption of Theorem \ref{FHFH} ($\sigma$-compactness, for example, does not satisfy   Mr\'{o}wka's condition (W); see  \cite{Ko3}). The purpose of this article is to prove the following version of Theorem \ref{FHFH}. Specific topological properties ${\mathcal P}$ which satisfy the requirements of Theorem  \ref{JJHJJ} below are $\sigma$-compactness, the Lindel\"{o}f property and the linearly Lindel\"{o}f property.  Note that in Theorem 3.19 of \cite{Ko2} we have shown that conditions (1) and (3) of Theorem \ref{JJHJJ} are equivalent, if  ${\mathcal P}$ is $\sigma$-compactness, and in Theorem 3.21 of \cite{Ko2} we have shown that conditions (1) and (2) of Theorem \ref{JJHJJ} are equivalent, if  ${\mathcal P}$ is the Lindel\"{o}f property. Thus, in some sense, Theorem \ref{JJHJJ} generalizes Theorems  3.19 and 3.21 of \cite{Ko2}, and at the same time, brings them  under a same umbrella.

\begin{theorem}\label{JJHJJ}
Let $X$ and $Y$  be locally compact paracompact spaces  and let  ${\mathcal P}$ be a closed hereditary topological property of compact spaces  which is
preserved under finite sums of subspaces and coincides with $\sigma$-compactness in the realm of locally compact paracompact spaces. The following are
equivalent:
\begin{itemize}
\item[\rm(1)] $\lambda_{{\mathcal P}} X\backslash X$ and $\lambda_{{\mathcal P}} Y\backslash Y$ are homeomorphic.
\item[\rm(2)] $({\mathscr E}^{\,C}_{\,{\mathcal P}}(X), \leq)$ and $({\mathscr E}^{\,C}_{\,{\mathcal P}}(Y), \leq)$ are order-isomorphic.
\item[\rm(3)] $({\mathscr E}^{\,*}_{\,local-{\mathcal P}}(X), \leq)$ and $({\mathscr E}^{\,*}_{\,local-{\mathcal P}}(Y), \leq)$ are order-isomorphic.
\end{itemize}
\end{theorem}

We now introduce some notation which will be widely used in this article.

\begin{notation}\label{GTD}
Let $X$ be a Tychonoff space $X$. For a subset $A$ of $X$ denote
\[A^*=\mbox{cl}_{\beta X}A\backslash X.\]
In particular, $X^*=\beta X\backslash X$.
\end{notation}

\begin{remark}\label{HUFR}
Note that the notation given in Notation \ref{GTD} can be ambiguous, as $A^*$ can mean either $\beta A\backslash A$ or $\mbox{cl}_{\beta X}A\backslash X$. However, since for $C^*$-embedded subsets these two notions coincide, this will not cause any confusion.
\end{remark}

\begin{definition}[\cite{HJW}]\label{KJH}
For a Tychonoff space $X$, let
\[\sigma X=\bigcup\{\mbox{cl}_{\beta X}H:H\subseteq X\mbox{ is $\sigma$-compact}\}.\]
\end{definition}

\begin{notation}\label{FTFE}
Let $X$ be a locally compact paracompact non-compact space. Then $X$ can be represented as
\[X=\bigoplus_{i\in I} X_i\]
for some index set $I$, with  each $X_i$ for $i\in I$, being $\sigma$-compact and non-compact (see Theorem 5.1.27 and Exercise 3.8.C of \cite{E}). For any $J\subseteq I$ denote
\[X_J=\bigcup_{i\in J}X_i.\]
Thus, using the notation of Notation \ref{GTD}, we have
\[X_J^*=\mbox{cl}_{\beta X}\Big(\bigcup_{i\in J}X_i\Big)\big\backslash X.\]
\end{notation}

\begin{remark}\label{FTDE}
Note that in Notation \ref{FTFE} the set $X_J^*$ is homeomorphic to $\beta X_J\backslash X_J$, as $\mbox{cl}_{\beta X}X_J$ is homeomorphic to $\beta X_J$ (see Corollary 3.6.8 of \cite{E}). Thus, when $J$ is countable (since $X_J$ is $\sigma$-compact and locally compact) $X_J^*$ is a zero-sets in $\mbox{cl}_{\beta X}X_J$ (see 1B of \cite{W}). But $\mbox{cl}_{\beta X}X_J$ is clopen in $\beta X$, as  $X_J$ is clopen in $X$ (see Corollary 3.6.5 of \cite{E}) therefore, $X_J^*$ is a zero-sets in $\beta X$.
Also, note that with the notation given in Notation \ref{FTFE}, we have
\[\sigma X=\bigcup\{\mbox{cl}_{\beta X} X_J:J\subseteq I\mbox{ is countable}\}.\]
Note that $\sigma X$ is open in $\beta X$ and it contains $X$.
\end{remark}

\section{Partially ordered set of one-point extensions as related to topologies of subspaces of outgrowth}

In Lemma \ref{JHV} we establish a connection between one-point Tychonoff extensions of a given space $X$ and compact non-empty
subsets of its outgrowth $X^*$. Lemma \ref{JHV} (and its preceding lemmas) is known (see e.g. \cite{MRW1}). It
is included here for the sake of completeness.

\begin{lemma}\label{KFH}
Let $X$ be a Tychonoff  space and let $C$ be a non-empty compact subset of $X^*$. Let $T$ be the space which is obtained from $\beta X$
by contracting $C$ to a point $p$. Then the subspace $Y=X\cup\{p\}$ of $T$ is Tychonoff and $\beta Y=T$.
\end{lemma}

\begin{proof}
Let $q:\beta X\rightarrow T$ be the quotient mapping. Note that $T$ is Hausdorff, and thus, being a continuous image of $\beta X$, it is compact.
Also, note that $Y$ is dense in $T$. Therefore, $T$ is a compactification of $Y$. To show that $\beta Y=T$, it suffices to verify that every continuous
$h:Y\rightarrow{\mathbf I}$ is continuously extendable over $T$. Let $h:Y\rightarrow{\mathbf I}$ be  continuous.  Let $G:\beta X\rightarrow{\mathbf I}$  continuously extend $hq|(X\cup C):X\cup C\rightarrow{\mathbf I}$
(note that $\beta(X\cup C)=\beta X$, as $X\subseteq X\cup C\subseteq \beta X$; see Corollary 3.6.9 of \cite{E}). Define  $H:T\rightarrow{\mathbf I}$ such that
$H|(\beta X\backslash C)=G|(\beta X\backslash C)$ and $H(p)=h(p)$. Then $H|Y=h$, and since $Hq=G$ is continuous, the function $H$ is continuous.
\end{proof}

\begin{notation}\label{KJB}
Let $X$ be a Tychonoff space and let $Y\in {\mathscr E}(X)$. Denote by
\[\tau_Y:\beta X\rightarrow \beta Y\]
the (unique) continuous extension of $\mbox{id}_X$.
\end{notation}

\begin{lemma}\label{JFV}
Let $X$ be a Tychonoff space and let $Y=X\cup\{p\}\in {\mathscr E}(X)$. Let $T$ be the
space which is obtained from $\beta X$ by contracting  $\tau_Y^{-1}(p)$ to the
point $p$, and let $q:\beta X\rightarrow T$ be the
quotient mapping. Then $T=\beta Y$  and $\tau_Y=q$.
\end{lemma}

\begin{proof}
We need to show that $Y$ is a subspace of $T$. Since $\beta Y$ is also a compactification of $X$ and
$\tau_Y|X=\mbox{id}_X$, by Theorem 3.5.7 of \cite{E}, we have  $\tau_Y(X^*)=\beta Y\backslash X$.  For an open subset $W$ of $\beta Y$, the set  $q(\tau_Y^{-1}(W))$ is open in $T$, as $q^{-1}(q(\tau_Y^{-1}(W)))=\tau_Y^{-1}(W)$ is open in $\beta X$. Therefore
\[Y\cap W=Y\cap q\big(\tau_Y^{-1}(W)\big)\]
is open in $Y$, when $Y$ is considered as  a subspace of $T$. For the converse, note that if $V$ is open in $T$, since
\[Y\cap V=Y\cap \big(\beta Y\backslash\tau_Y\big(\beta X\backslash q^{-1}(V)\big)\big)\]
and  $\tau_Y(\beta X\backslash q^{-1}(V))$ is compact and thus closed in $\beta Y$, the set $Y\cap V$ is open in $Y$ in its original topology.
By Lemma \ref{KFH} we have $T=\beta Y$. This
also implies that $\tau_Y=q$, as $\tau_Y,q:\beta X\rightarrow\beta Y$ are
continuous and coincide with $\mbox{id}_X$ on the dense subset $X$ of $\beta X$.
\end{proof}

\begin{lemma}\label{DFH}
Let $X$ be a Tychonoff space. Let $Y_i\in {\mathscr E}(X)$, where $i=1,2$, and denote by $\tau_i=\tau_{Y_i}:\beta X\rightarrow \beta Y_i$ the continuous
extension of $\mbox{\em id}_X$. The following are equivalent:
\begin{itemize}
\item[\rm(1)] $Y_1\leq Y_2$.
\item[\rm(2)] $\tau^{-1}_2(Y_2\backslash X)\subseteq\tau^{-1}_1(Y_1\backslash X)$.
\end{itemize}
\end{lemma}

\begin{proof}
Let $Y_i=X\cup\{p_i\}$ where $i=1,2$.  (1) {\em  implies} (2). Suppose that (1) holds. By definition, there exists a continuous
$f:Y_2\rightarrow Y_1$ such that $f|X= \mbox{id}_X$. Let $f_\beta:\beta Y_2\rightarrow \beta Y_1$  continuously  extend $f$. Note that the continuous
functions $f_\beta\tau_2,\tau_1:\beta X\rightarrow \beta Y_1$ coincide  with $\mbox{id}_X$ on the dense subset $X$ of $\beta X$, and thus
$f_\beta\tau_2= \tau_1$. Note that $X$ is dense in $\beta Y_i$ (where $i=1,2$), as it is  dense in $Y_i$, and therefore,  $\beta Y_i$ is a compactification
of $X$. Since $f_\beta|X=\mbox{id}_X$, by Theorem 3.5.7 of \cite{E}, we have $f_\beta(\beta Y_2\backslash X)=\beta Y_1\backslash X$, and thus $f_\beta(p_2)\in \beta Y_1\backslash X$. But $f_\beta(p_2)=f(p_2)$, which implies that $f_\beta(p_2)\in Y_1\backslash X=\{p_1\}$.  Therefore
\begin{eqnarray*}
\tau^{-1}_2(p_2)&\subseteq&\tau^{-1}_2\big(f_\beta^{-1}\big(f_\beta(p_2)\big)\big)\\&=&(f_\beta\tau_2)^{-1}\big(f_\beta(p_2)\big)
=\tau_1^{-1}\big(f_\beta(p_2)\big)=\tau_1^{-1}(p_1).
\end{eqnarray*}

(2) {\em  implies} (1).  Suppose that (2) holds. Let  $f:Y_2\rightarrow Y_1$ be defined  such that $f(p_2)=p_1$ and $f|X=\mbox{id}_X$.
We show that $f$ is continuous, this will show that $Y_1\leq Y_2$. Note that by Lemma \ref{JFV}, the space $\beta Y_2$ is the quotient space of $\beta X$ which is obtained
by contracting  $\tau^{-1}_2(p_2)$  to a point, and $\tau_2$ is its corresponding quotient mapping. Thus, in particular, $Y_2$ is the quotient space of
$X\cup \tau^{-1}_2 (p_2)$, and therefore, to show that $f$ is continuous, it suffices to show that $f\tau_2|(X\cup\tau^{-1}_2 (p_2))$ is continuous. We show this by verifying that
$f\tau_2(t)=\tau_1(t)$ for each $t\in X\cup\tau^{-1}_2 (p_2)$. This obviously holds if $t\in X$. If $t\in \tau^{-1}_2 (p_2)$, then $\tau_2(t)=p_2$, and thus   $f\tau_2 (t)=p_1$. But since $t\in \tau^{-1}_2(\tau_2 (t))$, we have $t\in \tau^{-1}_1(p_1)$, and therefore $\tau_1(t)=p_1$. Thus
$f\tau_2(t)=\tau_1(t)$ in this case as well.
\end{proof}

\begin{lemma}\label{JHV}
Let $X$ be a Tychonoff space. Define a function
\[\Theta:\big({\mathscr E}(X),\leq\big)\rightarrow\big(\{C\subseteq X^*:C \mbox{ is  compact} \}\backslash\{\emptyset\},\subseteq\big)\]
by
\[\Theta(Y)=\tau^{-1}_Y(Y\backslash X)\]
for any $Y\in{\mathscr E}(X)$. Then $\Theta$ is an anti-order-isomorphism.
\end{lemma}

\begin{proof}
To show that  $\Theta$ is well-defined, let $Y\in {\mathscr E}(X)$. Note that since $X$ is  dense in $Y$, the space  $X$ is dense in  $\beta Y$. Thus   $\tau_Y:\beta X\rightarrow \beta Y$ is onto, as $\tau_Y(\beta X)$ is a compact (and therefore closed) subset of $\beta Y$ and it contains $X=\tau_Y(X)$. Thus  $\tau^{-1}_Y(Y\backslash X)\neq\emptyset$. Also,  since $\tau_Y|X=\mbox{id}_X$ we have $\tau^{-1}_Y(Y\backslash X)\subseteq X^*$, and since the singleton $Y\backslash X$ is closed in $\beta Y$, its inverse image  $\tau^{-1}_Y(Y\backslash X)$ is closed in $\beta X$, and therefore  it is compact.
Now we show that $\Theta$ is onto,  Lemma \ref{DFH} will then complete the proof. Let $C$  be a non-empty  compact subset of $X^*$. Let $T$ be the quotient space of $\beta X$ which is obtained by contracting $C$ to a point $p$. Consider the subspace $Y=X\cup\{p\}$  of $T$. Then
$Y\in {\mathscr E}(X)$, and thus, by Lemma \ref{KFH} we have $\beta Y=T$. The quotient mapping $q:\beta X\rightarrow T$ is identical to $\tau_Y$, as it coincides with $\mbox{id}_X$ on the dense subset $X$ of $\beta X$. Therefore
\[\Theta(Y)=\tau^{-1}_Y(p)=q^{-1}(p)=C.\]
\end{proof}

\begin{notation}\label{JLLHV}
For a  Tychonoff space $X$ denote by
\[\Theta_X:\big({\mathscr E}(X),\leq\big)\rightarrow\big(\{C\subseteq X^*:C \mbox{ is compact} \}\backslash\{\emptyset\},\subseteq\big)\]
the anti-order-isomorphism defined by
\[\Theta_X(Y)=\tau^{-1}_Y(Y\backslash X)\]
for any $Y\in {\mathscr E}(X)$.
\end{notation}

Lemmas \ref{JYYH} and \ref{HDJYYH} below are known results (see \cite{Ko2}).

\begin{lemma}\label{JYYH}
Let $X$ be a Tychonoff  space. For a $Y\in {\mathscr E}(X)$ the following are equivalent:
\begin{itemize}
\item[\rm(1)] $Y\in {\mathscr E}^{\,*}(X)$.
\item[\rm(2)] $\Theta_X(Y)\in {\mathscr Z}(\beta X)$.
\end{itemize}
\end{lemma}

\begin{proof}
Let $Y=X\cup\{p\}$. (1) {\em  implies} (2).  Suppose that (1) holds. Let $\{V_n:n\in\mathbf{N}\}$ be an open base at $p$ in $Y$. For each
$n\in \mathbf{N}$, let $V'_n$ be an open subset of $\beta Y$ such that $Y\cap V'_n=V_n$, and let $f_n:\beta Y\rightarrow\mathbf{I}$ be continuous and such that $f_n(p)=0$ and $f_n(\beta Y\backslash V'_n)\subseteq\{1\}$. Let
\[Z=\bigcap_{n=1}^\infty Z(f_n)\in {\mathscr Z}(\beta Y).\]
We show that $Z=\{p\}$. Obviously, $p\in Z$. Let $t\in Z$ and suppose to the contrary that $t\neq p$. Let $W$ be an open neighborhood of $p$ in $\beta Y$ such that
$t\notin \mbox{cl}_{\beta Y}W$. Then $Y\cap W$ is an open neighborhood of $p$ in $Y$. Let  $k\in \mathbf{N}$ be such that $V_k\subseteq Y\cap W$. We have
\begin{eqnarray*}
t\in Z(f_k)\subseteq V'_k &\subseteq&\mbox{cl}_{\beta Y}V'_k\\&=&\mbox{cl}_{\beta Y}(Y\cap V'_k)\\&=&\mbox{cl}_{\beta Y}V_k\subseteq\mbox{cl}_{\beta Y}(Y\cap W)\subseteq\mbox{cl}_{\beta Y}W
\end{eqnarray*}
which is a contradiction. This shows that  $t=p$ and therefore $Z\subseteq \{p\}$. Thus $\{p\}=Z\in {\mathscr Z}(\beta Y)$, which implies that
$\tau^{-1}_Y(p)\in {\mathscr Z}(\beta X)$.

(2) {\em  implies} (1).  Suppose that (2) holds. Let $\tau^{-1}_Y(p)=Z(f)$ where $f:\beta X\rightarrow\mathbf{I}$ is continuous.  Note that by
Lemma \ref{JFV} the space $\beta Y$ is obtained from $\beta X$ by contracting $\tau^{-1}_Y(p)$ to $p$ with $\tau_Y:\beta X\rightarrow \beta Y$ as the quotient
mapping. Then for each $n\in \mathbf{N}$ the set $\tau_Y(f^{-1}([0,1/n)))$ is an open neighborhood of $p$ in $\beta Y$. We show that the collection
\[\big\{Y\cap\tau_Y\big(f^{-1}\big([0,1/n)\big)\big):n\in \mathbf{N}\big\}\]
of open neighborhoods of $p$ in $Y$ constitutes an open base at $p$ in $Y$, this will show (1).  Let $V$ be an open neighborhood of $p$ in $Y$.
Let $V'$ be an open subset of $\beta Y$ such that $Y\cap V'=V$. Then $p\in V'$ and thus
\[\bigcap_{n=1}^\infty f^{-1}\big([0,1/n]\big)=Z(f)=\tau_Y^{-1}(p)\subseteq\tau_Y^{-1}(V').\]
By compactness we have $f^{-1}([0,1/k])\subseteq\tau_Y^{-1}(V')$ for some  $k\in \mathbf{N}$. Therefore
\begin{eqnarray*}
Y\cap\tau_Y\big(f^{-1}\big([0,1/k)\big)\big)&\subseteq& Y\cap\tau_Y\big(f^{-1}\big([0,1/k]\big)\big)\\&\subseteq& Y\cap\tau_Y\big(\tau_Y^{-1}(V')\big)\subseteq Y\cap V'=V.
\end{eqnarray*}
\end{proof}

\begin{lemma}\label{HDJYYH}
Let $X$ be a locally compact  space. For a $Y\in {\mathscr E}(X)$ the following are equivalent:
\begin{itemize}
\item[\rm(1)] $Y\in {\mathscr E}^{\,C}(X)$.
\item[\rm(2)] $\Theta_X(Y)\in {\mathscr Z}(X^*)$.
\end{itemize}
\end{lemma}

\begin{proof}
Let $Y=X\cup\{p\}$. (1) {\em  implies} (2).  Suppose that (1) holds. Then $Y^*$ is an
$F_\sigma$ in $\beta Y$. Let $Y^*=\bigcup_{n=1}^\infty K_n$ where each $K_n$ is closed in $\beta Y$ for $n\in\mathbf{N}$. Then
\[X^*=\tau^{-1}_Y(p)\cup\bigcup_{n=1}^\infty K_n\]
(recall that $\beta Y$ is the quotient space of $\beta X$ which is obtained by contracting $\tau^{-1}_Y(p)$ to $p$ and $\tau_Y$ is its quotient mapping; see
Lemma \ref{JFV}). For each  $n\in \mathbf{N}$, let  $f_n:\beta X\rightarrow \mathbf{I}$ be continuous and  such that
\[f_n\big(\tau^{-1}_Y(p)\big)=\{0\}\mbox{ and }f_n(K_n)\subseteq\{1\}.\]
Let $f=\sum_{n=1}^\infty f_n/2^n$. Then $f:\beta X\rightarrow \mathbf{I}$ is continuous and
\[\tau^{-1}_Y(p)=Z(f)\cap X^*\in {\mathscr Z}(X^*).\]

(2) {\em  implies} (1).  Suppose that (2) holds. Let  $\tau^{-1}_Y(p)=Z(g)$ where  $g:X^*\rightarrow \mathbf{I}$ is continuous. Then, using  Lemma \ref{JFV}, we have
\begin{eqnarray*}
Y^*=X^*\backslash\tau^{-1}_Y(p)&=&X^*\backslash Z(g)\\&=&g^{-1}\big((0,1]\big)=\bigcup_{n=1}^\infty g^{-1}\big([1/n,1]\big)
\end{eqnarray*}
and  each  set $g^{-1}([1/n,1])$, for $n\in \mathbf{N}$, being closed in $X^*$, is compact (note that since $X$  is locally compact,
$X^*$ is compact) and thus closed in $\beta Y$.  Therefore, $Y^*$ is an $F_\sigma$ in $\beta Y$, that is,
$Y$ is  \v{C}ech-complete.
\end{proof}

The following lemma justifies our requirement on ${\mathcal P}$ in Theorem \ref{HJHGF}. We simply need $\lambda_{{\mathcal P}} X$ to have a more familiar structure.

\begin{lemma}\label{GFFYH}
Let  ${\mathcal P}$ be a topological property which is preserved under finite closed sums of subspaces. The  following are equivalent:
\begin{itemize}
\item[\rm(1)] The topological property ${\mathcal P}$ coincides with $\sigma$-compactness in the realm of locally compact paracompact spaces.
\item[\rm(2)] For every locally compact paracompact space $X$ we have
  \[\lambda_{{\mathcal P}} X=\sigma X.\]
\end{itemize}
\end{lemma}

\begin{proof}
(1) {\em  implies} (2).  Suppose that (1) holds. Let $X$ be a locally compact paracompact space. Assume the notation of Notation  \ref{FTFE}. Let $J\subseteq I$ be countable. Then $X_J$ is $\sigma$-compact and thus (since it is also locally compact and paracompact) it has
${\mathcal P}$. Note that $X_J$ is clopen in $X$ thus it has a clopen  closure in $\beta X$, therefore
\[\mbox{cl}_{\beta X}X_J=\mbox{int}_{\beta X}\mbox{cl}_{\beta X}X_J\subseteq\lambda_{{\mathcal P}} X\]
that is, $\sigma X\subseteq\lambda_{{\mathcal P}} X$. To see the reverse inclusion, let $C\in Coz(X)$ be such that $\mbox{cl}_X C$  has ${\mathcal P}$.
Then  (since $\mbox{cl}_X C$ being closed in $X$ is also locally compact and paracompact)  $\mbox{cl}_X C$ is $\sigma$-compact. Therefore
\[\mbox{int}_{\beta X}\mbox{cl}_{\beta X}C\subseteq\mbox{cl}_{\beta X}C\subseteq\sigma X\]
which shows that $\lambda_{{\mathcal P}} X\subseteq\sigma X$. Thus  $\lambda_{{\mathcal P}} X=\sigma X$.

(2) {\em  implies} (1).  Suppose that (2) holds. Let $X$ be a locally compact paracompact space. By assumption we have $\lambda_{{\mathcal P}} X=\sigma X$. We verify that $X$ has ${\mathcal P}$ if and only if $X$ is $\sigma$-compact. Assume the notation of Notation  \ref{FTFE}. Suppose that $X$ has ${\mathcal P}$. Then $\lambda_{{\mathcal P}} X=\beta X$ and thus  $\sigma X=\beta X$.  Now, by compactness, we have
\[\beta X=\mbox{cl}_{\beta X}X_{J_1}\cup\cdots\cup\mbox{cl}_{\beta X}X_{J_n}\]
for some $n\in \mathbf{N}$ and some countable $J_1, \ldots,  J_n\subseteq I$. Therefore
\[X=X_{J_1} \cup\cdots\cup X_{J_n}\]
is $\sigma$-compact. For the converse, suppose that $X$ is $\sigma$-compact. Then $\sigma X=\beta X$ and (since $\lambda_{{\mathcal P}} X=\sigma X$)
we have  $\beta X=\lambda_{{\mathcal P}} X$. Thus, by compactness, we have
\[\beta X=\mbox{int}_{\beta X}\mbox{cl}_{\beta X}C_1\cup\cdots\cup\mbox{int}_{\beta X}\mbox{cl}_{\beta X}C_n\]
for some  $n\in \mathbf{N}$ and some $C_1,\ldots, C_n\in Coz(X)$ such that  $\mbox{cl}_X C_i$ has ${\mathcal P}$ for any $i=1,\ldots,n$. Now, using our
assumption, the space
\[X=\mbox{cl}_XC_1\cup\cdots\cup\mbox{cl}_XC_n\]
being a finite union of its closed   ${\mathcal P}$-subspaces, has  ${\mathcal P}$.
\end{proof}

\begin{lemma}\label{HJGHG}
Let $X$ be a locally compact paracompact space  and let  ${\mathcal P}$ be a closed hereditary topological property of compact spaces  which is
preserved under finite sums of subspaces and coincides with $\sigma$-compactness in the realm of locally compact paracompact spaces. For a
$Y\in{\mathscr E}(X)$ the following are equivalent:
\begin{itemize}
\item[\rm(1)] $Y\in {\mathscr E}^{\,C}_{\,{\mathcal P}}(X)$.
\item[\rm(2)] $\Theta_X(Y)\in {\mathscr Z}(X^*)$ and $\beta X\backslash\lambda_{{\mathcal P}} X\subseteq\Theta_X(Y)$.
\end{itemize}
Thus, in particular
\[\Theta_X\big({\mathscr E}^{\,C}_{\,{\mathcal P}}(X)\big)=\big\{Z\in{\mathscr Z}(X^*): \beta X\backslash\lambda_{{\mathcal P}} X\subseteq Z\big\}\backslash\{\emptyset\}.\]
\end{lemma}

\begin{proof}
Let $Y=X\cup\{p\}$. (1) {\em  implies} (2).  Suppose that (1) holds. By Lemma \ref{HDJYYH} we have  $\tau^{-1}_Y(p)\in {\mathscr Z}(X^*)$. Note
that by Lemma  \ref{GFFYH} we have  $\lambda_{{\mathcal P}} X=\sigma X$. Let $t\in \beta X\backslash\sigma X$ and suppose to the contrary that
$t\notin \tau^{-1}_Y(p)$.  Let $f:\beta X\rightarrow \mathbf{I}$ be continuous and such that $f(t)=0$ and $f(\tau^{-1}_Y(p))=\{1\}$. Since
$\tau_Y(f^{-1}([0,1/2]))$ is compact, the set
\[T=X\cap f^{-1}\big([0,1/2]\big)=Y\cap\tau_Y\big(f^{-1}\big([0,1/2]\big)\big)\]
being closed in $Y$, has ${\mathcal P}$. But $T$, being closed in $X$, is locally compact and paracompact, and thus, having  ${\mathcal P}$, it is
$\sigma$-compact. Therefore, by the definition of $\sigma X$ we have $\mbox{cl}_{\beta X}T\subseteq\sigma X$. But since
\begin{eqnarray*}
t\in f^{-1}\big([0,1/2)\big)&\subseteq&\mbox{cl}_{\beta X}f^{-1}\big([0,1/2)\big)\\&=&\mbox{cl}_{\beta X}\big(X\cap f^{-1}\big([0,1/2)\big)\big)\\&\subseteq&\mbox{cl}_{\beta X}\big(X\cap f^{-1}\big([0,1/2]\big)\big)=\mbox{cl}_{\beta X}T
\end{eqnarray*}
we have $t\in \sigma X$, which contradicts  the choice of $t$.  Thus $t\in \tau^{-1}_Y(p)$ and therefore $\beta X\backslash\sigma X\subseteq\tau^{-1}_Y(p)$.

(2) {\em  implies} (1).  Suppose that (2) holds. Note that since $X$ is locally compact, the set $X^*$ is closed in
(the normal space) $\beta X$ and thus, since $\tau^{-1}_Y(p)\in {\mathscr Z}(X^*)$ (using the Tietze-Urysohn Theorem) we have
$\tau^{-1}_Y(p)=Z\cap X^*$ for some $Z\in {\mathscr Z}(\beta X)$.  Note that by Lemma  \ref{GFFYH} we have $\lambda_{{\mathcal P}} X=\sigma X$.
Now, since $\beta X\backslash\sigma X\subseteq\tau^{-1}_Y(p)\subseteq Z$ we have $\beta X\backslash Z\subseteq\sigma X$. Therefore, assuming the notation of
Notation  \ref{FTFE} (since $\beta X\backslash Z$, being a cozero-set in $\beta X$, is  $\sigma$-compact)  we have
\[\beta X\backslash Z\subseteq\bigcup_{n=1}^\infty\mbox{cl}_{\beta X}X_{J_n}\subseteq\mbox{cl}_{\beta X}X_J\]
where $J_1,J_2,\ldots\subseteq I$ are countable and $J=J_1\cup J_2\cup\cdots$.
But
\[Y=\tau_Y(Z)\cup(X\backslash Z)\subseteq \tau_Y(Z)\cup X_J\]
and thus we have
\begin{equation}\label{HFUY}
Y=\tau_Y(Z)\cup X_J.
\end{equation}
Now, since  $X_J$ has ${\mathcal P}$, as it is $\sigma$-compact (and  being closed in $X$, it is locally compact and  paracompact)
and $\tau_Y(Z)$ has  ${\mathcal P}$, as it is compact, from   (\ref{HFUY}) it follows that the space $Y$, being a finite union of its
${\mathcal P}$-subspaces, has  ${\mathcal P}$. The fact that $Y$ is \v{C}ech-complete follows from Lemma \ref{HDJYYH}.
\end{proof}

The following generalizes Lemma 3.18 of \cite{Ko2}.

\begin{lemma}\label{HDGF}
Let $X$ be a locally compact paracompact space  and let  ${\mathcal P}$ be a closed hereditary topological property of compact spaces
which is preserved under finite sums of subspaces and coincides with $\sigma$-compactness in the realm of locally compact paracompact spaces.
For a $Y\in {\mathscr E}(X)$ the following are equivalent:
\begin{itemize}
\item[\rm(1)] $Y\in {\mathscr E}^{\,*}_{\,local-{\mathcal P}}(X)$.
\item[\rm(2)] $\Theta_X(Y)\in {\mathscr Z}(\beta X)$ and $\Theta_X(Y)\subseteq\lambda_{{\mathcal P}} X$.
\end{itemize}
Thus, in particular
\[\Theta_X\big({\mathscr E}^{\,*}_{\,local-{\mathcal P}}(X)\big)=\big\{Z\in{\mathscr Z}(\beta X): Z\subseteq\lambda_{{\mathcal P}} X\backslash X\big\}\backslash\{\emptyset\}.\]
\end{lemma}

\begin{proof}
Let $Y=X\cup\{p\}$. (1) {\em  implies} (2).  Suppose that (1) holds. Since $Y\in {\mathscr E}^{\,*}(X)$, by Lemma \ref{JYYH} we have
$\tau^{-1}_Y(p)\in {\mathscr Z}(\beta X)$. Let  $\tau^{-1}_Y(p)=Z(f)$ for some continuous $f:\beta X\rightarrow\mathbf{I}$.
Since $Y$ is locally-${\mathcal P}$, there exists an open neighborhood $V$ of $p$ in $Y$ such that  $\mbox{cl}_Y V$ has  ${\mathcal P}$.
Let $V'$ be an open subset of $\beta Y$ such that  $Y\cap V'=V$. Then $p\in V'$, and thus since
\[\bigcap_{n=1}^\infty f^{-1}\big([0,1/n]\big)=Z(f)=\tau^{-1}_Y(p)\subseteq\tau^{-1}_Y(V')\]
by compactness, we have $f^{-1}([0,1/k])\subseteq\tau^{-1}_Y(V')$ for some $k\in\mathbf{N}$. Now, for each $n\geq k$, since
\begin{eqnarray*}
&&Y\cap\tau_Y\big(f^{-1}\big([0,1/n]\big)\backslash f^{-1}\big(\big[0,1/(n+1)\big)\big)\big)\\&\subseteq& Y\cap\tau_Y\big(f^{-1}\big([0,1/k]\big)\big)\subseteq Y\cap\tau_Y\big(\tau_Y^{-1}(V')\big)\subseteq Y\cap V'=V\subseteq\mbox{cl}_Y V
\end{eqnarray*}
the set
\begin{eqnarray*}
K_n&=&X\cap\big(f^{-1}\big([0,1/n]\big)\backslash f^{-1}\big(\big[0,1/(n+1)\big)\big)\big)\\&=&Y\cap\tau_Y\big(f^{-1}\big([0,1/n]\big)\backslash f^{-1}\big(\big[0,1/(n+1)\big)\big)\big)
\end{eqnarray*}
being closed in $\mbox{cl}_Y V$, has ${\mathcal P}$, and therefore (since being closed in $X$ it is locally compact and paracompact) it is
$\sigma$-compact. (It might be helpful to recall that by Lemma \ref{JFV} the space $\beta Y$ is obtained from $\beta X$ by contracting
$\tau^{-1}_Y(p)$ to $p$ with $\tau_Y$ as its quotient mapping.) Thus, the set
\[X\cap f^{-1}\big([0,1/k]\big)=\bigcup_{n=k}^\infty K_n\]
is  $\sigma$-compact, and therefore, by the definition of  $\sigma X$, we have
\[\mbox{cl}_{\beta X}\big(X\cap f^{-1}\big([0,1/k]\big)\big)\subseteq\sigma X.\]
But
\begin{eqnarray*}
Z(f)\subseteq f^{-1}\big([0,1/k)\big)&\subseteq&\mbox{cl}_{\beta X} f^{-1}\big([0,1/k)\big)\\&=&\mbox{cl}_{\beta X}\big(X\cap f^{-1}\big([0,1/k)\big)\big)\\&\subseteq&\mbox{cl}_{\beta X}\big(X\cap f^{-1}\big([0,1/k]\big)\big)
\end{eqnarray*}
from which it follows that  $\tau^{-1}_Y(p)\subseteq \sigma X$. Finally, note that by Lemma \ref{GFFYH} we have $\lambda_{{\mathcal P}} X=\sigma X$.

(2) {\em  implies} (1).  Suppose that (2) holds. By Lemma \ref{JYYH} we have $Y\in {\mathscr E}^{\,*}(X)$. Therefore, it suffices to verify that $Y$ is
locally-${\mathcal P}$. Also, since by assumption $X$ is  locally compact, it is  locally-${\mathcal P}$, as ${\mathcal P}$ is assumed to be a topological property of compact spaces. Thus, we need only to verify that $p$ has  an open neighborhood in $Y$ whose closure in $Y$ has  ${\mathcal P}$. Let $g:\beta X\rightarrow\mathbf{I}$ be continuous
and such that $Z(g)=\tau^{-1}_Y(p)$. Then since
\[\bigcap_{n=1}^\infty g^{-1}\big([0,1/n]\big)=Z(g)\subseteq \lambda_{{\mathcal P}} X\]
by compactness (and since $\lambda_{{\mathcal P}} X$ is open in $\beta X$) we have $g^{-1}([0,1/k])\subseteq \lambda_{{\mathcal P}} X$ for
some $k\in \mathbf{N}$. Note that  by Lemma  \ref{GFFYH} we have $\lambda_{{\mathcal P}} X=\sigma X$. Assume  the notation of Notation  \ref{FTFE}.
By compactness, we have
\[g^{-1}\big([0,1/k]\big)\subseteq\mbox{cl}_{\beta X}X_{J_1}\cup\cdots\cup\mbox{cl}_{\beta X}X_{J_n}=\mbox{cl}_{\beta X}X_J\]
where $n\in \mathbf{N}$, the sets  $J_1,\ldots, J_n\subseteq I$ are countable and $J=J_1\cup\cdots\cup J_n$. The set $X\cap g^{-1}([0,1/k])\subseteq X_J$, being closed in the latter ($\sigma$-compact space) is $\sigma$-compact, and therefore (since being closed in $X$, it is locally compact and paracompact)
it has  ${\mathcal P}$. Let
\[V=Y\cap \tau_Y\big(g^{-1}\big([0,1/k)\big)\big).\]
Then $V$ is an open neighborhood of $p$ in $Y$. We show that  $\mbox{cl}_Y V$ has  ${\mathcal P}$. But this follows, since
\begin{eqnarray*}
\mbox{cl}_Y V\subseteq Y\cap\tau_Y\big(g^{-1}\big([0,1/k]\big)\big)&=&\big(X\cap\tau_Y\big(g^{-1}\big([0,1/k]\big)\big)\big)\cup\{p\}\\&=&\big(X\cap g^{-1}\big([0,1/k]\big)\big)\cup\{p\}
\end{eqnarray*}
and the latter, being a finite union of its ${\mathcal P}$-subspaces (note that the singleton $\{p\}$, being compact, has ${\mathcal P}$) has
${\mathcal P}$, and  thus, its closed subset $\mbox{cl}_Y V$, also has ${\mathcal P}$.
\end{proof}

Lemmas \ref{KHSA}--\ref{KHJGD} are from \cite{Ko1}.

\begin{lemma}\label{KHSA}
Let $X$ be a locally compact paracompact space. If $Z\in{\mathscr Z}(\beta X)$ in non-empty then $Z\cap\sigma X\neq\emptyset$
\end{lemma}

\begin{proof}
Let $\{x_n\}_{n=1}^\infty$ be a sequence in $\sigma X$. Assume the notation of  Notation  \ref{FTFE}.
Then $\{x_n:n\in \mathbf{N}\}\subseteq\mbox{cl}_{\beta X}X_J$ for some countable $J\subseteq I$. Therefore
$\{x_n:n\in \mathbf{N}\}$ has a limit point in $\mbox{cl}_{\beta X}X_J\subseteq\sigma X$. Thus $\sigma X$ is countably compact,
and therefore is pseudocompact, and $\upsilon(\sigma X)=\beta(\sigma X)=\beta X$ (note that the latter equality holds, as $X\subseteq\sigma X\subseteq\beta X$). The result now follows, as for any Tychonoff space $T$, any non-empty
zero-set of $\upsilon T$ meets $T$  (see Lemma 5.11 (f) of \cite{PW}).
\end{proof}

\begin{lemma}\label{JGRD}
Let $X$ be a locally compact paracompact space. If $Z\in  {\mathscr Z}(X^*)$ is non-empty then $Z\cap\sigma X\neq\emptyset$.
\end{lemma}

\begin{proof}
Let $S\in {\mathscr Z}(\beta X)$ be such that $S\cap X^*=Z$ (which exists, as $X^*$ is closed in (the normal space) $\beta X$, as $X$ is locally compact, and thus, by the Tietze-Urysohn Theorem, every continuous function from
$X^*$ to $\mathbf{I}$ is continuously extendible over $\beta X$). By Lemma \ref{KHSA} we have $S\cap\sigma X\neq\emptyset$. Suppose that
$S\cap (\sigma X\backslash X)=\emptyset$. Then $S\cap\sigma X=X\cap S$. Assume the notation of Notation  \ref{FTFE}. Let
$J=\{i\in I: X_i\cap S\neq\emptyset\}$. Then $J$ is finite. Note that since $X_J$ is clopen in $X$, it has a clopen  closure in
$\beta X$. Now
\[T=S\cap (\beta X\backslash\mbox{cl}_{\beta X}X_J)\in  {\mathscr Z}(\beta X)\]
misses $\sigma X$, and therefore, by Lemma \ref{KHSA}  we have $T=\emptyset$. But this is a contradiction, as $Z=S\cap(\beta X\backslash \sigma X)\subseteq T$.
This shows that
\[Z\cap(\sigma X\backslash X)=S\cap (\sigma X\backslash X)\neq\emptyset.\]
\end{proof}

\begin{lemma}\label{KHJGD}
Let $X$ be a locally compact paracompact space. For any $S,T\in{\mathscr Z}(X^*)$, if $S\cap\sigma X\subseteq T\cap\sigma X$ then
$S\subseteq T$.
\end{lemma}

\begin{proof}
Suppose to the contrary that $S\backslash T\neq\emptyset$. Let $s\in S\backslash T$. Let $f:\beta X\rightarrow\mathbf{I}$ be continuous and such that
$f(s)=0$ and $f(T)\subseteq\{1\}$. Then $Z(f)\cap S$ is non-empty, and thus by Lemma \ref{JGRD} it follows that $Z(f)\cap S\cap\sigma X\neq\emptyset$. But this is not
possible, as
\[Z(f)\cap S\cap\sigma X\subseteq Z(f)\cap T=\emptyset.\]
\end{proof}

The following lemma is from \cite{Ko2}.

\begin{lemma}\label{FSFH}
Let $X$ and $Y$ be locally compact spaces. The following are equivalent:
\begin{itemize}
\item[\rm(1)] $X^*$ and $Y^*$ are homeomorphic.
\item[\rm(2)] $({\mathscr E}^{\,C}(X),\leq)$ and $({\mathscr E}^{\,C}(Y),\leq)$ are order-isomorphic.
\end{itemize}
\end{lemma}

\begin{proof}
This follows from the fact that in  a compact space the order-structure of the set of its all zero-sets (partially ordered with $\subseteq$) determines its topology.
\end{proof}

The proof of the following theorem is essentially a combination of the proofs we have given for Theorems 3.19 and 3.21 in \cite{Ko2} with the appropriate usage of the preceding lemmas. The reasonably detailed proof is included here for the reader's convenience.

\begin{theorem}\label{HJHGF}
Let $X$ and $Y$  be locally compact paracompact (non-compact) spaces  and let  ${\mathcal P}$ be a closed hereditary topological property of compact spaces  which is preserved under finite sums of subspaces and coincides with $\sigma$-compactness in the realm of locally compact paracompact spaces. The following are equivalent:
\begin{itemize}
\item[\rm(1)] $\lambda_{{\mathcal P}} X\backslash X$ and $\lambda_{{\mathcal P}} Y\backslash Y$ are homeomorphic.
\item[\rm(2)] $({\mathscr E}^{\,C}_{\,{\mathcal P}}(X), \leq)$ and $({\mathscr E}^{\,C}_{\,{\mathcal P}}(Y), \leq)$ are order-isomorphic.
\item[\rm(3)] $({\mathscr E}^{\,*}_{\,local-{\mathcal P}}(X), \leq)$ and $({\mathscr E}^{\,*}_{\,local-{\mathcal P}}(Y), \leq)$ are order-isomorphic.
\end{itemize}
\end{theorem}

\begin{proof}
Let
\[X=\bigoplus_{i\in I} X_i\mbox{ and }Y=\bigoplus_{j\in J} Y_j\]
for some index sets $I$ and $J$ with  each $X_i$ and $Y_j$ for $i\in I$ and $j\in J$ being $\sigma$-compact and non-compact. We will use notation of Notation \ref{FTFE} and Remark \ref{FTDE} without reference. Note that by Lemma  \ref{GFFYH} we have $\lambda_{{\mathcal P}} X=\sigma X$ and $\lambda_{{\mathcal P}} Y=\sigma Y$. Let
\[\omega\sigma X=\sigma X\cup\{\Omega\}\mbox{ and } \omega\sigma Y=\sigma Y\cup\{\Omega'\}\]
denote the one-point compactifications of $\sigma X$ and $\sigma Y$, respectively.

(1) {\em  implies} (2).  Suppose that (1) holds. Suppose that either $X$ or $Y$, say $X$, is $\sigma$-compact. Then $\sigma Y\backslash Y$ is compact, as it is homeomorphic to $\sigma X\backslash X=X^*$, and the latter is compact, as $X$ is locally compact. Thus
\[\sigma Y\backslash Y=Y_{H_1}^*\cup\cdots\cup Y_{H_n}^*= Y_H^*\]
where $n\in \mathbf{N}$, the sets  $H_1,\ldots,H_n\subseteq J$ are countable and
\[H=H_1\cup\cdots\cup H_n.\]
Now, if there exists some
$u\in J\backslash H$, then
since $Y_u\cap Y_H=\emptyset$ we have
\[\mbox{cl}_{\beta Y}Y_u\cap\mbox{cl}_{\beta Y}Y_H=\emptyset.\]
Therefore $\mbox{cl}_{\beta Y}Y_u\subseteq Y$, contradicting the fact that $Y_u$ is non-compact. Thus $J=H$ and   $Y$ is  $\sigma$-compact. Therefore
$\sigma Y\backslash Y=Y^*$. Note that by Lemmas \ref{HDJYYH} and \ref{HJGHG} we have ${\mathscr E}^{\,C}_{\,{\mathcal P}}(X)={\mathscr E}^{\,C}(X)$ and ${\mathscr E}^{\,C}_{\,{\mathcal P}}(Y)={\mathscr E}^{\,C}(Y)$. The result now follows from Lemma \ref{FSFH}.

Suppose that $X$ and $Y$ are non-$\sigma$-compact. Let $f:\sigma X\backslash X\rightarrow\sigma Y\backslash Y$ denote a
homeomorphism. We define an order-isomorphism
\[\phi:\big(\Theta_X\big({\mathscr E}^{\,C}_{\,{\mathcal P}}(X)\big),\subseteq\big)\rightarrow\big(\Theta_Y\big({\mathscr E}^{\,C}_{\,{\mathcal P}}(Y)\big),\subseteq\big).\]
Since $\Theta_X$ and $\Theta_Y$ are
anti-order-isomorphisms, this will prove (2). Let $D\in\Theta_X({\mathscr E}^{\,C}_{\,{\mathcal P}}(X))$. By Lemma \ref{HJGHG} we have $D\in
{\mathscr Z}(X^*)$ and $\beta X\backslash\sigma X\subseteq D$. Since $X^*\backslash D\subseteq\sigma X$, being a cozero-set in $X^*$ is $\sigma$-compact, there exists a countable
$G\subseteq I$ such that $X^*\backslash D\subseteq X_G^*$. Now, since $D\cap X_G^*\in {\mathscr Z}(X_G^*)$, we have
\[f(D\cap X_G^*)\in {\mathscr Z}\big(f(X_G^*)\big).\]
Since $X_G^*$ is open in $\sigma
X\backslash X$, its homeomorphic image  $f(X_G^*)$ is open in $\sigma Y\backslash
Y$, and thus, is open in $Y^*$. But $f(X_G^*)$ is compact, as it is a continuous image of a compact space, and therefore, $f(X_G^*)$ is clopen in $Y^*$. Thus
\[f(D\cap X_G^*)\cup\big(Y^*\backslash f(X_G^*)\big)\in {\mathscr Z}(Y^*).\]
Let
\[\phi(D)=f\big(D\cap(\sigma X\backslash X)\big)\cup (\beta Y\backslash\sigma Y).\]
Note that since
\begin{eqnarray*}
f\big(D\cap(\sigma X\backslash X)\big)&=&f\big((D\cap X_G^*)\cup\big((\sigma X\backslash X)\backslash X_G^*\big)\big)\\&=&f(D\cap X_G^*)\cup\big((\sigma Y\backslash Y)\backslash f(X_G^*)\big)
\end{eqnarray*}
we have
\begin{eqnarray*}
\phi(D)&=&f\big(D\cap(\sigma X\backslash X)\big)\cup(\beta Y\backslash\sigma Y)\\&=&f(D\cap X_G^*)\cup\big((\sigma Y\backslash Y)\backslash f(X_G^*)\big)\cup(\beta Y\backslash\sigma Y)\\&=&f(D\cap X_G^*)\cup\big(Y^*\backslash f(X_G^*)\big)
\end{eqnarray*}
which shows that $\phi$ is well-defined. The function $\phi$ is clearly an order-homomorphism. Since $f^{-1}:\sigma Y\backslash Y\rightarrow\sigma X\backslash X$ also is a homeomorphism, as above, it induces an order-homomorphism
\[\psi:\big(\Theta_Y\big({\mathscr E}^{\,C}_{\,{\mathcal P}}(Y)\big),\subseteq\big)\rightarrow\big(\Theta_X\big({\mathscr E}^{\,C}_{\,{\mathcal P}}(X)\big),\subseteq\big)\]
which is defined by
\[\psi(D)=f^{-1}\big(D\cap(\sigma Y\backslash Y)\big)\cup  (\beta X\backslash\sigma X)\]
for any $D\in\Theta_Y({\mathscr E}^{\,C}_{\,{\mathcal P}}(Y))$. It is now easy to see that $\psi=\phi^{-1}$, which shows that $\phi$ is an order-isomorphism.

(2) {\em  implies} (1).  Suppose that (2) holds. Suppose that either $X$ or $Y$, say $X$,  is $\sigma$-compact (and non-compact). Then
$\sigma X= \beta X$, and thus, by Lemmas \ref{HDJYYH} and \ref{HJGHG}, we have ${\mathscr E}^{\,C}_{\,{\mathcal P}}(X)={\mathscr E}^{\,C}(X)$. Suppose that $Y$ is
non-$\sigma$-compact. Note that  $X$, being paracompact and non-compact, is non-pseudocompact (see Theorems 3.10.21, 5.1.5 and 5.1.20 of \cite{E}) and therefore, $X^*$ contains at least two
elements, as almost compact spaces are pseudocompact (see Problem 5U (1) of \cite{PW}; recall that a Tychonoff space $T$ is called {\em almost compact} if $\beta T\backslash T$ has at most one element). Thus, there exist two disjoint non-empty zero-sets of  $X^*$ corresponding to two elements in ${\mathscr E}^{\,C}(X)$  with no
common upper bound in   ${\mathscr E}^{\,C}(X)$. But this is not true, as ${\mathscr E}^{\,C}(X)$ is order-isomorphic to ${\mathscr E}^{\,C}_{\,{\mathcal P}}(Y)$,
and any two elements in the latter have a common upper bound in ${\mathscr E}^{\,C}_{\,{\mathcal P}}(Y)$.
(Note that since  $Y$ is non-$\sigma$-compact, the set $\beta Y\backslash\sigma Y$ is non-empty, and by Lemma \ref{HJGHG}, the image  of any element
in ${\mathscr E}^{\,C}_{\,{\mathcal P}}(Y)$ under $\Theta_Y$ contains $\beta Y\backslash\sigma Y$.) Therefore, $Y$ also is  $\sigma$-compact and by Lemmas \ref{HDJYYH} and \ref{HJGHG}, we have  ${\mathscr E}^{\,C}_{\,{\mathcal P}}(Y)={\mathscr E}^{\,C}(Y)$. Now, since $\sigma Y= \beta Y$, the result follows from Lemma \ref{FSFH}.

Next, suppose that $X$ and $Y$ are both non-$\sigma$-compact. We show that the two compact spaces $\omega\sigma X\backslash X$ and
$\omega\sigma Y\backslash Y$ are homeomorphic, by showing that their corresponding sets of zero-sets (partially ordered with $\subseteq$) are order-isomorphic. Since  $\Theta_X$ and $\Theta_Y$ are
anti-order-isomorphisms, condition (2) implies the existence of an order-isomorphism
\[\phi:\big(\Theta_X\big({\mathscr E}^{\,C}_{\,{\mathcal P}}(X)\big),\subseteq\big)\rightarrow\big(\Theta_Y\big({\mathscr E}^{\,C}_{\,{\mathcal P}}(Y)\big),\subseteq\big).\]
We define an order-isomorphism
\[\psi:\big({\mathscr Z}(\omega\sigma X\backslash X),\subseteq\big)\rightarrow\big({\mathscr Z}(\omega\sigma Y\backslash Y),\subseteq\big)\]
as follows. Let $Z\in{\mathscr Z}(\omega\sigma X\backslash X)$. Suppose that
$\Omega\in Z$. Then, since $(\omega\sigma X\backslash X)\backslash
Z$ is a cozero-set in (the compact space) $\omega\sigma X\backslash X$, it is
$\sigma$-compact. Thus $(\omega\sigma X\backslash X)\backslash Z\subseteq X_G^*$
for some countable $G\subseteq I$. Since $X_G^*$ is clopen in
$X^*$, we have
\[\big(Z\backslash\{\Omega\}\big)\cup(\beta X\backslash\sigma X)=(Z\cap X_G^*)\cup(X^*\backslash X_G^*)\in{\mathscr Z}(X^*).\]
In this case, we let
\[\psi(Z)=\big(\phi\big(\big(Z\backslash\{\Omega\}\big)\cup(\beta X\backslash\sigma X)\big)\backslash(\beta Y\backslash\sigma Y)\big)\cup\{\Omega'\}.\]
Now, suppose that  $\Omega\notin Z$. Then $Z\subseteq \sigma X\backslash
X$, and therefore  $Z\subseteq X_G^*$ for some
countable $G\subseteq I$, and thus, using this, one can write
\begin{equation}\label{KJUJB}
Z=X^*\backslash\bigcup_{n=1}^\infty Z_n\mbox{ where }\beta X\backslash\sigma X\subseteq Z_n\in{\mathscr Z}(X^*)\mbox{ for any } n\in \mathbf{N}.
\end{equation}
In this case, we let
\[\psi(Z)=Y^*\backslash\bigcup_{n=1}^\infty\phi (Z_n).\]
We check that $\psi$ is well-defined. Assume the representation given in (\ref{KJUJB}).
Since  $Y^*\backslash\phi(Z_n)\subseteq \sigma Y$ for all $n\in \mathbf{N}$, there exists a countable $H\subseteq
J$ such that $Y^*\backslash\phi(Z_n)\subseteq Y_H^*$ for all $n\in \mathbf{N}$.

\begin{xclaim}
For a $Z\in{\mathscr Z}(\omega\sigma X\backslash X)$ with $\Omega\notin Z$ assume the representation given in (\ref{KJUJB}). Let $H\subseteq J$ be countable and such that $Y^*\backslash\phi(Z_n)\subseteq Y_H^*$ for all $n\in \mathbf{N}$. Let $A$ be such that
$\phi (A)=Y^*\backslash Y_H^*$. Then
\[Y^*\backslash\bigcup_{n=1}^\infty \phi(Z_n)=\phi(A\cup Z)\backslash\phi(A).\]
\end{xclaim}

\subsubsection*{Proof of the claim} Suppose that $y\in Y^*$ and $y\notin\phi (Z_n)$ for all $n\in \mathbf{N}$. If
$y\notin\phi (A\cup Z)\backslash \phi (A)$,  then since $y\notin
\phi(Z_1)\supseteq \phi(A)$  we have $y\notin \phi (A\cup Z)$. Therefore, there exists some $B\in {\mathscr Z}(Y^*)$ containing $y$
such that $B\cap \phi (A\cup Z)=\emptyset$ and $B\cap \phi (Z_n)=\emptyset$ for all $n\in \mathbf{N}$. Let $C$ be such that
$\phi(C)=B\cup\phi (A\cup Z)$, and let  $S_n$ for any $n\in \mathbf{N}$, be
such that
\begin{eqnarray*}
\phi(S_n)&=&\phi(C)\cap\phi(Z_n)\\&=&\big(B\cup\phi(A\cup Z)\big)\cap\phi(Z_n)\\&=&\big(B\cap\phi(Z_n)\big)\cup\big(\phi(A\cup Z)\cap\phi(Z_n)\big)=\phi(A\cup Z)\cap\phi(Z_n).
\end{eqnarray*}
Since $A\subseteq Z_n$, as $\phi(A)\subseteq \phi(Z_n)$ and $Z\cap
Z_n=\emptyset$, we have $A\cap Z=\emptyset$, which implies that
\[(A\cup Z)\cap Z_n=(A\cap Z_n)\cup(Z\cap Z_n)=A\]
for all $n\in \mathbf{N}$. Clearly $S_n\subseteq(A\cup Z)\cap Z_n$, as by above $\phi(S_n)\subseteq\phi(A\cup Z)$ and $\phi(S_n)\subseteq\phi(Z_n)$ for any $n\in\mathbf{N}$. Thus $\phi(S_n)\subseteq\phi(A)$ for all $n\in \mathbf{N}$.
But since $\phi(A)\subseteq\phi(Z_n)$, we have $\phi(A)\subseteq\phi(S_n)$, and therefore
\[\phi(C\cap Z_n)\subseteq \phi (C)\cap\phi (Z_n)=\phi(S_n)=\phi(A)\]
for any $n\in \mathbf{N}$. This implies that $C\cap Z_n\subseteq A$ for all $n\in \mathbf{N}$. Thus
\[C\backslash Z=C\cap\bigcup_{n=1}^\infty Z_n=\bigcup_{n=1}^\infty (C\cap Z_n)\subseteq A.\]
Therefore $C\subseteq A\cup Z$ and
we have $B\subseteq\phi (C)\subseteq\phi(A\cup Z)$, which is a contradiction, as $B\cap\phi(A\cup Z)=\emptyset$. This shows that
\[Y^*\backslash\bigcup_{n=1}^\infty\phi(Z_n)\subseteq\phi (A\cup Z)\backslash \phi(A).\]
Now, suppose that $y\in\phi(A\cup Z)\backslash\phi(A)$. Suppose to the contrary
that $y\in\phi(Z_n)$ for some $n\in \mathbf{N}$. Then
\[y\in\phi(Z_n)\cap\phi(A\cup Z)=\phi(D)\]
for some $D$. Clearly $D\subseteq Z_n$ and $D\subseteq A\cup Z$, as $\phi(D)\subseteq\phi (Z_n)$ and $\phi(D)\subseteq\phi(A\cup Z)$. This implies that \[D\subseteq Z_n\cap( A\cup Z)=(Z_n\cap A)\cup (Z_n\cap Z)=Z_n\cap A\subseteq A\]
and thus $y\in\phi (A)$, as $\phi(D)\subseteq\phi(A)$, which is a contradiction. This proves the claim.

\noindent Now, suppose that
\[Z=X^*\backslash\bigcup_{n=1}^\infty S_n\mbox{ and }Z=X^*\backslash \bigcup_{n=1}^\infty Z_n\]
are two representations for some $Z\in{\mathscr Z}(\omega\sigma X\backslash X)$
with  $\Omega\notin Z$ such that each $S_n,Z_n\in {\mathscr Z}(X^*) $ contains $\beta X\backslash\sigma X$ for $n\in \mathbf{N}$. Choose a countable
$H\subseteq J$ such that
\[Y^*\backslash\phi (S_n)\subseteq Y_H^*\mbox{ and }Y^*\backslash\phi (Z_n)\subseteq Y_H^*\]
for all $n\in \mathbf{N}$. Then, by the claim, we have
\[Y^*\backslash\bigcup_{n=1}^\infty \phi(S_n)=\phi (A\cup
Z)\backslash\phi(A)=Y^*\backslash\bigcup_{n=1}^\infty \phi(Z_n)\]
where $A$ is such that $\phi(A)=Y^*\backslash Y_H^*$. This shows that $\psi$ is well-defined.
Next, we show that $\psi$ is an order-isomorphism. Suppose that  $S,Z\in{\mathscr Z}(\omega\sigma X\backslash X)$  and
$S\subseteq Z$. We consider the following cases.

\begin{description}
  \item[Case 1] Suppose that $\Omega\in S$. Then $\Omega\in Z$, and clearly
\begin{eqnarray*}
\psi(S)&=&\big(\phi\big(\big(S\backslash\{\Omega\}\big)\cup(\beta X\backslash\sigma X)\big)\backslash(\beta Y\backslash\sigma Y)\big)\cup\{\Omega'\}\\&\subseteq&\big(\phi\big(\big(Z\backslash\{\Omega\}\big)\cup(\beta X\backslash\sigma X)\big)\backslash(\beta Y\backslash\sigma Y)\big)\cup\{\Omega'\}=\psi(Z).
\end{eqnarray*}
  \item[Case 2] Suppose that $\Omega\notin S$ but $\Omega\in Z$. Let
\[E=\phi\big(\big(Z\backslash\{\Omega\}\big)\cup(\beta X\backslash \sigma X)\big)\]
and let
\[S=X^*\backslash\bigcup_{n=1}^\infty S_n\]
where each $S_n\in{\mathscr Z}(X^*)$ contains $\beta
X\backslash\sigma X$ for $n\in \mathbf{N}$. Clearly $Y^*\backslash E\subseteq \sigma Y$.
Let $H\subseteq J$ be countable and such that $Y^*\backslash\phi(S_n)\subseteq Y_H^*$ for all $n\in \mathbf{N}$ and $Y^*\backslash E\subseteq Y_H^*$. By the claim, we have $\psi(S)=\phi(A\cup S)\backslash\phi (A)$, where
$\phi (A)=Y^*\backslash Y_H^*$. Since
$Y^*\backslash Y_H^*\subseteq E$, we have
\[A\subseteq\big(Z\backslash\{\Omega\}\big)\cup(\beta X\backslash\sigma X).\]
Now
\[\psi(S)=\phi(A\cup S)\backslash\phi(A)\subseteq\phi(A\cup S)\subseteq\phi\big(\big(Z\backslash\{\Omega\}\big)\cup(\beta X\backslash\sigma
X)\big)\]
which implies that
\[\psi(S)\subseteq\big(\phi\big(\big(Z\backslash\{\Omega\}\big)\cup(\beta X\backslash\sigma X)\big)\backslash(\beta Y\backslash\sigma Y)\big)\cup\{\Omega'\}=\psi(Z).\]
  \item[Case 3] Suppose that $\Omega\notin Z$. Then  $\Omega\notin S$. Let
\[S=X^*\backslash\bigcup_{n=1}^\infty S_n\mbox{ and }Z=X^*\backslash\bigcup_{n=1}^\infty Z_n\]
where each $S_n, Z_n\in {\mathscr Z}(X^*)$ contains $\beta X\backslash\sigma X$ for $n\in \mathbf{N}$. Clearly
\[S=S\cap Z=\Big(X^*\backslash\bigcup_{n=1}^\infty S_n\Big)\cap\Big(X^*\backslash\bigcup_{n=1}^\infty Z_n\Big)=
X^*\backslash\bigcup_{n=1}^\infty(S_n\cup Z_n)\]
and thus, since $\phi(Z_n)\subseteq\phi(S_n\cup Z_n)$ for all $n\in \mathbf{N}$, it follows that
\[\psi(S)=Y^*\backslash\bigcup_{n=1}^\infty \phi(S_n\cup
Z_n)\subseteq Y^*\backslash\bigcup_{n=1}^\infty \phi(Z_n)=\psi(Z).\]
\end{description}

\noindent Note that since
\[\phi^{-1}:\big(\Theta_Y\big({\mathscr E}^{\,C}_{\,{\mathcal P}}(Y)\big),\subseteq\big)\rightarrow\big(\Theta_X\big({\mathscr E}^{\,C}_{\,{\mathcal P}}(X)\big),\subseteq\big)\]
also is an order-isomorphism, as above, it induces an order-isomorphism
\[\gamma:\big({\mathscr Z}(\omega\sigma Y\backslash Y),\subseteq\big)\rightarrow\big({\mathscr Z}(\omega\sigma X\backslash X),\subseteq\big)\]
which is easy to see that $\gamma=\psi^{-1}$. Therefore, $\psi$ is an order-isomorphism.
It then follows that there exists a homeomorphism $f:\omega\sigma X\backslash X\rightarrow\omega\sigma Y\backslash
Y$ such that $f(Z)=\psi(Z)$, for any $Z\in
{\mathscr Z}(\omega\sigma X\backslash X)$. Now since for each
countable $G\subseteq I$ we have
\[f(X_G^*)=\psi(X_G^*)\subseteq\sigma Y\backslash Y\]
it follows that $f(\sigma X\backslash X)=\sigma Y\backslash Y$. Thus $\sigma X\backslash X$ and $\sigma Y\backslash Y$  are
homeomorphic.

(1) {\em  implies} (3). Suppose that (1) holds.  Suppose that either $X$ or $Y$, say $X$,  is $\sigma$-compact. Then $\sigma X=\beta X$ and thus,
arguing as in part (1)$\Rightarrow$(2),
it follows that $Y$ also is $\sigma$-compact. Therefore $\sigma Y=\beta Y$.  Note that by  Lemmas \ref{JYYH} and \ref{HDGF}  we have ${\mathscr E}^{\,*}_{\,local-{\mathcal P}}(X)={\mathscr E}^{\,*}(X)$ and since $X^*\in {\mathscr Z}(\beta X)$  (as $X$  is $\sigma$-compact and locally compact; see 1B of \cite{W}) by Lemmas \ref{JYYH} and  \ref{HDJYYH} we have  ${\mathscr E}^{\,*}(X)={\mathscr E}^{\,C}(X)$. Thus   ${\mathscr E}^{\,*}_{\,local-{\mathcal P}}(X)={\mathscr E}^{\,C}(X)$ and similarly  ${\mathscr E}^{\,*}_{\,local-{\mathcal P}}(Y)={\mathscr E}^{\,C}(Y)$. The result now follows from Lemma \ref{FSFH}.

Suppose that $X$ and $Y$ are non-$\sigma$-compact. Let  $f:\sigma
X\backslash X\rightarrow \sigma Y\backslash Y$ be a homeomorphism. We define an order-isomorphism
\[\phi:\big(\Theta_X \big({\mathscr E}^{\,*}_{\,local-{\mathcal P}}(X)\big),\subseteq\big)\rightarrow\big(\Theta_Y\big({\mathscr E}^{\,*}_{\,local-{\mathcal P}}(Y)\big),\subseteq\big)\]
as follows. Let $Z\in\Theta_X ({\mathscr E}^{\,*}_{\,local-{\mathcal P}}(X))$. By Lemma \ref{HDGF} we have  $Z\in{\mathscr Z}(\beta X)$ and $Z\subseteq\sigma X\backslash X$. Thus $Z\subseteq X_G^*$ for some countable $G\subseteq I$. Now $f(Z)\in {\mathscr Z}(\sigma Y\backslash Y)$ and since $f(Z)$ is compact, as it is a continuous image of a compact space, it follows that $f(Z)\subseteq Y_H^*$ for some countable $H\subseteq J$. Therefore $f(Z)\in{\mathscr Z}(Y_H^*)$ and then $f(Z)\in{\mathscr Z}(\mbox{cl}_{\beta Y} Y_H)$. Since $\mbox{cl}_{\beta Y} Y_H$ is clopen in $\beta Y$ we have $f(Z)\in {\mathscr Z}(\beta Y)$. Define
\[\phi(Z)=f(Z).\]
It is obvious that $\phi$ is an order-homomorphism. If we let
\[\psi:\big(\Theta_Y\big({\mathscr E}^{\,*}_{\,local-{\mathcal P}}(Y)\big),\subseteq\big)\rightarrow\big(\Theta_X \big({\mathscr E}^{\,*}_{\,local-{\mathcal P}}(X)\big),\subseteq\big)\]
be defined by
\[\psi(Z)=f^{-1}(Z)\]
for any $Z\in\Theta_Y({\mathscr E}^{\,*}_{\,local-{\mathcal P}}(Y))$, then $\psi=\phi^{-1}$ which shows that  $\phi$ is an order-isomorphism.

(3) {\em  implies} (1). Suppose that (3) holds. Suppose that either $X$ or $Y$, say $X$, is $\sigma$-compact (and non-compact). Then $\sigma X=\beta X$, and thus, by Lemmas \ref{JYYH} and \ref{HDGF}, we have ${\mathscr E}^{\,*}_{\,local-{\mathcal P}}(X)={\mathscr E}^{\,*}(X)$. Therefore, since $X^*\in {\mathscr Z}(\beta X)$ the set ${\mathscr E}^{\,*}_{\,local-{\mathcal P}}(X)$ has the smallest  element (namely, its one-point compactification $\omega X$). Thus ${\mathscr E}^{\,*}_{\,local-{\mathcal P}}(Y)$ also has the smallest element; denote this element by $T$.  Then, for each countable $H\subseteq J$ we have
\[Y_H^*\in \Theta_Y\big({\mathscr E}^{\,*}_{\,local-{\mathcal P}}(Y)\big)\]
and therefore $\sigma Y\backslash Y\subseteq\Theta_Y(T)$. By Lemma \ref{KHJGD}
(with $\Theta_Y(T)$ and $Y^*$ as the zero-sets in its statement) we have $Y^*\subseteq\Theta_Y(T)$. This
implies that $Y^*\in {\mathscr Z}(\beta Y)$ which shows that $Y$ is $\sigma$-compact. Thus $\sigma Y=\beta Y$, and by Lemmas \ref{JYYH} and \ref{HDGF}, we have ${\mathscr E}^{\,*}_{\,local-{\mathcal P}}(Y)={\mathscr E}^{\,*}(Y)$. Therefore, in this case  (and since by Lemmas \ref{JYYH} and \ref{HDJYYH} we have  ${\mathscr E}^{\,*}(X)={\mathscr E}^{\,C}(X)$ and ${\mathscr E}^{\,*}(Y)={\mathscr E}^{\,C}(Y)$) the result follows from Lemma \ref{FSFH}.

Next, suppose that  $X$ and $Y$ are both non-$\sigma$-compact. Since
$\Theta_X$ and $\Theta_Y$ are both anti-order-isomorphisms, there exists an order-isomorphism
\[\phi:\big(\Theta_X \big({\mathscr E}^{\,*}_{\,local-{\mathcal P}}(X)\big),\subseteq\big)\rightarrow\big(\Theta_Y\big({\mathscr E}^{\,*}_{\,local-{\mathcal P}}(Y)\big),\subseteq\big).\]
We extend $\phi$ by letting $\phi(\emptyset)=\emptyset$. We define a function
\[\psi:\big({\mathscr Z}(\omega\sigma X\backslash X),\subseteq\big)\rightarrow\big({\mathscr Z}(\omega\sigma Y\backslash Y),\subseteq\big)\]
and verify that it is an order-isomorphism. Let $Z\in {\mathscr Z}(\omega\sigma X\backslash X)$ with $\Omega\notin
Z$. Since $Z\subseteq X_G^*$ for some countable
$G\subseteq I$, we have $Z\in {\mathscr Z}(\beta X)$, and therefore
\[Z\in \Theta_X \big({\mathscr E}^{\,*}_{\,local-{\mathcal P}}(X)\big)\cup\{\emptyset\}.\]
In this
case, let
\[\psi(Z)=\phi(Z).\]
Now, suppose that $Z\in {\mathscr Z}(\omega\sigma X\backslash X)$ and
$\Omega\in Z$. Then $(\omega\sigma X\backslash X)\backslash Z$ is a cozero-set in $\omega\sigma X\backslash X$, and we have
\begin{equation}\label{UHFE}
Z=(\omega\sigma X\backslash X)\backslash\bigcup_{n=1}^\infty Z_n\mbox{ where }Z_n\in{\mathscr Z}(\omega\sigma X\backslash X)\mbox{ for any }n\in \mathbf{N}.
\end{equation}
Thus, as above, it follows that
\[Z_n\in\Theta_X\big({\mathscr E}^{\,*}_{\,local-{\mathcal P}}(X)\big)\cup\{\emptyset\}\]
for any $n\in \mathbf{N}$. We verify that
\begin{equation}\label{UGHDW}
\bigcup_{n=1}^\infty \phi (Z_n)\in Coz(\omega\sigma Y\backslash Y).
\end{equation}
To show this, note that since $\phi(Z_n)\subseteq\sigma Y\backslash Y$ there exists a countable
$H\subseteq J$ such that $\phi(Z_n)\subseteq Y_H^*$ for all $n\in\mathbf{N}$.

\begin{xclaim}
For a $Z\in {\mathscr Z}(\omega\sigma X\backslash X)$ with $\Omega\in Z$ assume the representation given in (\ref{UHFE}).
Let $H\subseteq J$ be countable and such that $\phi(Z_n)\subseteq Y_H^*$ for all $n\in \mathbf{N}$. Let $A$ be such that
$\phi (A)=Y_H^*$. Then
\[\phi(A\cap Z)=\phi(A)\backslash\bigcup_{n=1}^\infty \phi (Z_n).\]
\end{xclaim}

\subsubsection*{Proof of the claim} For each
$n\in \mathbf{N}$, since $A\cap Z\cap Z_n=\emptyset$, we have $\phi(A\cap Z)\cap
\phi(Z_n)=\emptyset$, as otherwise, $\phi(A\cap Z)$ and $\phi(Z_n)$ will have a common lower bound in $\Theta_Y({\mathscr E}^{\,*}_{\,local-{\mathcal P}}(Y))$, that is, $\phi(A\cap Z)\cap\phi(Z_n)$, whereas $A\cap Z$ and $Z_n$ do not have. Also $\phi(A\cap Z)\subseteq\phi(A)$. Therefore
\[\phi(A\cap Z)\subseteq\phi(A)\backslash\bigcup_{n=1}^\infty\phi(Z_n).\]
To show the reverse inclusion, let $y\in\phi(A)$ be such that $y\notin\phi(Z_n)$ for all $n\in\mathbf{N}$. There
exists some $B\in{\mathscr Z}(\beta Y)$ such that
$y\in B$ and  $B\cap\phi(Z_n)=\emptyset$ for all $n\in\mathbf{N}$. If
$y\notin\phi(A\cap Z)$, then there exists some $C\in{\mathscr Z}(\beta Y)$ such that $y\in C$ and
$C\cap\phi(A\cap Z)=\emptyset$. Let $D=\phi(A)\cap B\cap C$ and let $E$ be such that $\phi (E)=D$. For each
$n\in \mathbf{N}$, since $\phi(E)\cap\phi(Z_n)=\emptyset$, we have $E\cap
Z_n=\emptyset$, and thus $E\subseteq Z$. On the other hand,
since $\phi(E)\subseteq\phi(A)$ we have  $E\subseteq A$, and therefore $E\subseteq
A\cap Z$. Thus $\phi (E)\subseteq\phi(A\cap Z)$, which
implies that $\phi(E)=\emptyset$, as $\phi(E)\subseteq C$. This
contradiction shows that $y\in\phi(A\cap Z)$, which proves the claim.

\noindent Let $A$ be such that
$\phi (A)=Y_H^*$. Now, $\phi(A\cap Z)\in{\mathscr Z}(\omega\sigma Y\backslash Y)$, as $\phi(A\cap Z)\subseteq\phi(A)$. By the claim we have
\begin{eqnarray*}
(\omega\sigma Y\backslash Y)\backslash\bigcup_{n=1}^\infty \phi
(Z_n)&=&\Big(\phi (A)\backslash\bigcup_{n=1}^\infty \phi (Z_n)\Big)\cup
\big((\omega\sigma Y\backslash Y)\backslash \phi (A)\big)\\&=&\phi (A\cap Z)\cup\big((\omega\sigma Y\backslash Y)\backslash
\phi(A)\big)\in{\mathscr Z}(\omega\sigma Y\backslash Y)
\end{eqnarray*}
and (\ref{UGHDW}) is verified. In this case, we let
\[\psi(Z)=(\omega\sigma Y\backslash Y)\backslash\bigcup_{n=1}^\infty \phi(Z_n).\]
Next, we show that $\psi$ is well-defined. Assume  that
\[Z=(\omega\sigma X\backslash X)\backslash\bigcup_{n=1}^\infty S_n\]
with $S_n\in {\mathscr Z}(\omega\sigma X\backslash X)$ for all $n\in \mathbf{N}$, is another
representation of $Z$. We need to show that
\begin{equation}\label{GFYDL}
\bigcup_{n=1}^\infty \phi
(Z_n)=\bigcup_{n=1}^\infty \phi (S_n).
\end{equation}
Without any loss of generality, suppose to the contrary  that there exists some $m\in \mathbf{N}$ and  $y\in
\phi(Z_m)$ such that $y\notin\phi(S_n)$ for all $n\in \mathbf{N}$. Then there exists some $A\in{\mathscr Z}(\beta Y)$ such that $y\in A$ and
$A\cap\phi (S_n)=\emptyset$ for all $n\in \mathbf{N}$. Consider
\[A\cap\phi(Z_m)\in \Theta_Y \big({\mathscr E}^{\,*}_{\,local-{\mathcal P}}(Y)\big).\]
Let $B$ be such that $\phi(B)=A\cap\phi(Z_m)$.  Since $\phi(B)\subseteq A$ we have $\phi(B)\cap\phi(S_n)=\emptyset$ from which it follows that $B\cap S_n=\emptyset$ for all $n\in \mathbf{N}$. But $B\subseteq Z_m$, as $\phi(B)\subseteq\phi(Z_m)$, and we have
\[B\subseteq\bigcup_{n=1}^\infty Z_n=\bigcup_{n=1}^\infty S_n\]
which implies that $B=\emptyset$. But this is a contradiction, as $\phi(B)\neq\emptyset$. Therefore (\ref{GFYDL}) holds, and thus $\psi$ is well-defined.
To prove that $\psi$ is an order-isomorphism, let $S,Z\in{\mathscr Z}(\omega\sigma X\backslash X)$ and $S\subseteq Z$. The case when $S=\emptyset$ holds trivially. Assume that $S\neq\emptyset$. We consider the following cases.

\begin{description}
   \item[Case 1]  Suppose that  $\Omega\notin Z$. Then  $\Omega\notin S$ and we have
\[\psi(S)=\phi(S)\subseteq\phi(Z)=\psi(Z).\]
   \item[Case 2] Suppose that $\Omega\notin S$ but  $\Omega\in Z$.  Let
\[Z=(\omega\sigma X\backslash X)\backslash\bigcup_{n=1}^\infty Z_n\]
with $Z_n\in{\mathscr Z}(\omega\sigma X\backslash X)$ for all $n\in\mathbf{N}$. Then, since $S\subseteq Z$ we have $S\cap
Z_n=\emptyset$, and therefore $\phi(S)\cap \phi(Z_n)=\emptyset$ for all $n\in \mathbf{N}$.
Thus
\[\psi(S)=\phi(S)\subseteq(\omega\sigma Y\backslash
Y)\backslash\bigcup_{n=1}^\infty \phi (Z_n)=\psi(Z).\]
   \item[Case 3] Suppose that $\Omega\in S$. Then $\Omega\in Z$. Let
\[S=(\omega\sigma X\backslash
X)\backslash\bigcup_{n=1}^\infty S_n \mbox { and }Z=(\omega\sigma X\backslash X)\backslash\bigcup_{n=1}^\infty Z_n\]
where $S_n,Z_n\in{\mathscr Z}(\omega\sigma X\backslash X)$ for all $n\in\mathbf{N}$. Therefore
\begin{eqnarray*}
S=S\cap Z&=&\Big((\omega\sigma X\backslash
X)\backslash\bigcup_{n=1}^\infty S_n\Big)\cap\Big((\omega\sigma X\backslash
X)\backslash\bigcup_{n=1}^\infty Z_n\Big)\\&=&(\omega\sigma X\backslash X)\backslash\bigcup_{n=1}^\infty (S_n\cup
Z_n).
\end{eqnarray*}
Thus, since $\phi(Z_n)\subseteq\phi(S_n\cup Z_n)$ for all $n\in \mathbf{N}$, we have
\[\psi(S)=(\omega\sigma Y\backslash
Y)\backslash\bigcup_{n=1}^\infty\phi(S_n\cup Z_n)\subseteq(\omega\sigma Y\backslash
Y)\backslash\bigcup_{n=1}^\infty \phi (Z_n)=\psi(Z).\]
\end{description}

\noindent This shows that $\psi$ is an order-homomorphism.
To show that $\psi$ is an order-isomorphism, we note that
\[\phi^{-1}:\big(\Theta_Y \big({\mathscr E}^{\,*}_{\,local-{\mathcal P}}(Y)\big),\subseteq\big)\rightarrow\big(\Theta_X\big({\mathscr E}^{\,*}_{\,local-{\mathcal P}}(X)\big),\subseteq\big)\]
is an order-isomorphism. Let
\[\gamma:\big({\mathscr Z}(\omega\sigma Y\backslash Y),\subseteq\big)\rightarrow\big({\mathscr Z}(\omega\sigma X\backslash X),\subseteq\big)\]
be the induced order-homomorphism which is defined as above.
Then it is straightforward to see that $\gamma=\psi ^{-1}$, that is,
$\psi$ is an order-isomorphism. This implies the existence of a homeomorphism
$f:\omega\sigma X\backslash X\rightarrow\omega\sigma Y\backslash Y$
such that $f(Z)=\psi (Z)$ for every $Z\in{\mathscr Z}(\omega\sigma
X\backslash X)$. Therefore, for any countable $G\subseteq I$,
since $X_G^*\in {\mathscr Z}(\omega\sigma X\backslash X)$, we have
\[f(X_G^*)=\psi (X_G^*)=\phi(X_G^*)\subseteq\sigma Y\backslash Y.\]
Thus $f(\sigma X\backslash X)\subseteq\sigma
Y\backslash Y$, which shows that $f(\Omega)=\Omega'$.  Therefore
$\sigma X\backslash X$ and $\sigma Y\backslash Y$ are homeomorphic.
\end{proof}

\begin{example}\label{POJH}
The Lindel\"{o}f property and the linearly Lindel\"{o}f property (besides $\sigma$-compactness itself) are examples of topological properties ${\mathcal P}$ satisfying the assumption of Theorem \ref{HJHGF}. To see this, let $X$ be a locally compact paracompact space. Assume a representation for $X$ as in Notation \ref{FTFE}. Recall that a Hausdorff space $X$ is said to be {\em linearly Lindel\"{o}f} \cite{G} provided that every linearly ordered (by set inclusion $\subseteq$) open cover of $X$ has a countable subcover, equivalently, if every uncountable subset of $X$ has a complete accumulation point in $X$. (Recall that a point $x\in X$ is called a {\em complete accumulation point} of a set $A\subseteq X$ if for every neighborhood $U$ of $x$ in $X$ we have $|U\cap A|=|A|$.) Note that if $X$ is non-$\sigma$-compact then (using the notation of Notation \ref{FTFE}) the set $I$ is uncountable. Let $A=\{x_i:i\in I\}$ where $x_i\in X_i$ for each $i\in I$. Then $A$ is an uncountable subset of $X$ without (even) accumulation points. Thus $X$ cannot be linearly Lindel\"{o}f as well. For the converse, note that if $X$ is not linearly Lindel\"{o}f, then, obviously, $X$ is not Lindel\"{o}f, and therefore, is non-$\sigma$-compact, as it is well-known that $\sigma$-compactness and the Lindel\"{o}f property coincide in the realm of locally compact paracompact spaces (this fact is evident from the representation given for $X$ in Notation \ref{FTFE}).
\end{example}

Theorem \ref{HJHGF} might leave the impression that $({\mathscr E}^{\,C}_{\,{\mathcal P}}(X), \leq)$ and $({\mathscr E}^{\,*}_{\,local-{\mathcal P}}(X),\leq)$ are order-isomorphic. The following is to settle this, showing that in most cases this indeed is not going to be the case.

\begin{theorem}\label{HFS}
Let $X$ be a locally compact paracompact (non-compact) space  and let  ${\mathcal P}$ be a closed hereditary topological property of compact spaces  which is preserved under finite sums of subspaces and coincides with $\sigma$-compactness in the realm of locally compact paracompact spaces. The  following are equivalent:
\begin{itemize}
\item[\rm(1)] $X$ is $\sigma$-compact.
\item[\rm(2)] $({\mathscr E}^{\,C}_{\,{\mathcal P}}(X),\leq)$ and $({\mathscr E}^{\,*}_{\,local-{\mathcal P}}(X),\leq)$ are order-isomorphic.
\end{itemize}
\end{theorem}

\begin{proof}
Since $X$ is locally compact, the set $X^*$ is closed in (the normal space) $\beta X$ and thus, using the Tietze-Urysohn Theorem, every  zero-set of  $X^*$ is extendible to a zero-set of $\beta X$. Now if $X$ is $\sigma$-compact (since $X$ is also locally compact) we have $X^*\in{\mathscr Z}(\beta X)$ and therefore every zero-set of $X^*$ is a zero-set of $\beta X$. Note that $\lambda_{{\mathcal P}} X=\sigma X=\beta X$. Thus using Lemmas
\ref{HJGHG} and \ref{HDGF} we have
\[\Theta_X\big({\mathscr E}^{\,C}_{\,{\mathcal P}}(X)\big)={\mathscr Z}(X^*)\backslash\{\emptyset\}=\Theta_X\big({\mathscr E}^{\,*}_{\,local-{\mathcal P}}(X)\big)\]
from which it follows that
\[{\mathscr E}^{\,C}_{\,{\mathcal P}}(X)={\mathscr E}^{\,*}_{\,local-{\mathcal P}}(X).\]

If $X$ is non-$\sigma$-compact, then any two elements of ${\mathscr E}^{\,C}_{\,{\mathcal P}}(X)$ has a common upper bound while this is not the case for ${\mathscr E}^{\,*}_{\,local-{\mathcal P}}(X)$. To see this, note that by Lemma \ref{HJGHG} the set $\Theta_X({\mathscr E}^{\,C}_{\,{\mathcal P}}(X))$ is closed under finite intersections (note that the finite intersections are non-empty, as they contain $\beta X\backslash\sigma X$ and the latter is non-empty, as $X$ is non-$\sigma$-compact) while there exist (at least) two elements in $\Theta_X({\mathscr E}^{\,*}_{\,local-{\mathcal P}}(X))$ with empty intersection; simply consider $X_{i}^*$ and $X_{j}^*$ for some distinct $i,j\in I$ (we are assuming the representation for $X$ given in Notation \ref{FTFE}).
\end{proof}

We conclude this article with the following.

\begin{project}\label{HYF}
Let $X$  be a (locally compact paracompact) space and let ${\mathcal P}$ be a (closed hereditary) topological property (of compact spaces which is preserved under finite sums of subspaces and coincides with $\sigma$-compactness in the realm of locally compact paracompact spaces). Explore the relationship between the order structures of $({\mathscr E}^{\,C}_{\,{\mathcal P}}(X), \leq)$ and $({\mathscr E}^{\,*}_{\,local-{\mathcal P}}(X), \leq)$.
\end{project}

\end{document}